\documentclass[11pt,leqno]{amsart}

\usepackage{amscd}
\usepackage{color}
\usepackage{amssymb}
\usepackage{amsfonts,dsfont}
\usepackage{latexsym}
\usepackage{url}
\usepackage{verbatim}

\usepackage[T1]{fontenc}

\usepackage[hypertexnames=false,
backref=page,
    pdfpagemode=UseNone,
    breaklinks=true,
    extension=pdf,
    colorlinks=true,
    linkcolor=blue,
    citecolor=blue,
    urlcolor=blue,
]{hyperref}

\renewcommand{\Im}{\mathsf{Im}\,}

\theoremstyle{plain}
\newtheorem{theorem}{Theorem}[section]
\newtheorem*{theorem*}{Theorem}
\newtheorem{definition}[theorem]{Definition}

\newtheorem{prop}[theorem]{Proposition}
\newtheorem{cor}[theorem]{Corollary}

\theoremstyle{definition}
\newtheorem{rem}[theorem]{Remark}
\newtheorem{ex}[theorem]{Example}

\allowdisplaybreaks[1]

\title{Isometric immersions of locally conformally K\"ahler manifolds}

\author{Daniele Angella}
\address[Daniele Angella]{
 Dipartimento di Matematica e Informatica ``Ulisse Dini''\\
 Universit\`a degli Studi di Firenze\\
 viale Morgagni 67/a\\
 50134 Firenze, Italy
}
\email{daniele.angella@gmail.com}
\email{daniele.angella@unifi.it}

\author{Michela Zedda}
\address[Michela Zedda]{Dipartimento di Scienze Matematiche, Fisiche ed Informatiche\\
Plesso Matematico e Informatico\\
Universit\`a  di Parma\\
Parco Area delle Scienze 53/A\\
43124 Parma, Italy
}
\email{michela.zedda@unipr.it}
\email{michela.zedda@gmail.com}

\thanks{The first-named author has been partially supported by Project PRIN "Variet\'a  reali e complesse: geometria, topologia e analisi armonica", by Project SIR2014 "Analytic aspects in complex and hypercomplex geometry" code RBSI14DYEB, and by the group GNSAGA of INdAM.
The second-named author has been partially supported by the group GNSAGA of INdAM}

\date{\today}

\begin{document}

\begin{abstract}
We investigate isometric immersions of locally conformally K\"ahler metrics into Hopf manifolds. In particular, we study Hopf-induced metrics on compact complex surfaces.
\end{abstract}

\maketitle

\section{Introduction}
The celebrated Kodaira Embedding Theorem gives geometric and cohomological conditions under which analytic geometry reduces to algebraic geometry. In fact, it gives a holomorphic embedding for Hodge manifolds (namely, K\"ahler manifolds whose K\"ahler class is rational or, equivalently, admitting a holomorphic line bundle with positive curvature) into some projective space; and then the Serre GAGA applies.
In general, such an embedding is not isometric: {\itshape e.g.} there are no complex curves of constant negative curvature in $\mathds{C}{\rm P}^n$. This latter issue occurs also in $\mathds C^n$, where moreover the maximum principle does not allow to embed compact complex manifolds.
When the Kodaira canonical maps are isometric embeddings, or more in general when we have a holomorphic isometric immersion into $\mathds{C}{\rm P}^n$ endowed with the Fubini-Study metric, then the metric is called {\em projectively-induced}.
In \cite{tian-jdg} Tian proved that projectively-induced K\"ahler metrics are dense in the space of polarized K\"ahler metrics in the $\mathcal{C}^2$-topology, later generalized by \cite{ruan} to the $\mathcal C^\infty$-topology. This means that K\"ahler metrics on Hodge manifolds can be approximated by projectively-induced metrics.

The problem of which (real analytic) K\"ahler manifolds admit an isometric immersion into $\mathds{C}{\rm P}^n$, or more in general into complex space forms (that is, finite or infinite dimensional K\"ahler manifolds of constant holomorphic sectional curvature) has been studied by E. Calabi \cite{calabi-ann}.
Here, the complete simply connected complex space forms, according to the sign of the holomorphic sectional curvature, are the complex Euclidean space $\mathds C^N$ endowed with its canonical flat metric, the complex projective space $\mathds{C}{\rm P}^N$ with the Fubini-Study metric, and the complex hyperbolic space $\mathds{C}{\rm H}^N$ with the hyperbolic metric, where $N\in\mathds{N}\cup\{\infty\}$.
More precisely, E. Calabi introduces the {\em diastasis function} (as a symmetrization of a polarization of a local real analytic K\"ahler potential) and gives an algebraic criterion for the immersion in terms of the coefficients of the power series expansion of (some function of) it. We refer to \cite{loi-zedda} for an updated panorama on the subject.

Inspired by E. Calabi's work, we would like to study isometric immersions in the {\em locally conformally K\"ahler} (lcK from now on) setting, see {\itshape e.g.} \cite{dragomir-ornea, ornea-survey}.
We recall that an lcK structure on a complex manifold $M$ is given by a Hermitian $(1,1)$-form $\omega$ and a closed $1$-form, called Lee form, such that $d\omega=\theta\wedge\omega$. On some covering $\tilde M$, the Lee form becomes exact, {\itshape i.e.} there exists a potential $h$ such that it yields $\tilde\theta=dh$, and then $\tilde{\omega}:=e^{-h}\omega$ is a K\"ahler metric on $\tilde M$, and the deck transformation group acts by homotheties.
So, lcK geometry can be interpreted as an "equivariant (homothetic) K\"ahler geometry" or a first specific non-K\"ahler setting (see Gray-Hervella classification).
LcK metrics play a role on compact complex surfaces (see {\itshape e.g.} \cite{vaisman-polito, belgun, brunella, pontecorvo} and references therein): with the exception of Inoue surfaces, every known compact complex surfaces admit lcK metrics; and in any case, a covering admitting a K\"ahler metric exists. 

In the lcK context, the analogue of the projective space is played by {\em Hopf manifolds} $H^n$, that is, manifolds whose universal covering is $\mathds C^n \setminus \{0\}$ and whose fundamental group equals the infinite cyclic group $\mathds{Z}$, see \cite{belgun, gauduchon-ornea}.
An analogue of the Kodaira embedding in this setting has been proven by L. Ornea and M. Verbitsky \cite{ornea-verbitsky-MathAnn, ornea-verbitsky-MathAnn2, ornea-verbitsky-JGP}. More precisely, they proved the folowing:
{\itshape let $(M,\omega,\theta)$ be an lcK manifold of $\dim_{\mathds C} M \geq 3$. If it has a proper positive lcK potential, then there exists a holomorphic embedding into a linear Hopf manifold $(H^n,\omega_H,\theta_0)$.}
Here, by lcK potential they mean that a K\"ahler covering has a potential on which the deck transformations act with the same homothety factors as on the metric; the potential is assumed to be  proper, compare with \cite{ornea-verbitsky-JGP}.
Vaisman manifolds correspond to an action of the Hopf deck transformations generated by a semisimple element.

For our purpose, we consider an $N$-dimensional classical Hopf manifold $\left(H^N_\lambda,\omega_H, \theta_0\right)$, with $H^N_\lambda=(\mathds C^N\setminus\{0\})/\mathds{Z}$, $\omega_H=\|z\|^{-2}\omega_0$, and $\theta_0=d\log\|z\|^{-2}$, where $\mathds Z$ is generated by $\lambda\cdot \mathrm{id}$ for $\lambda\in\mathds{C}\setminus\{0\}$ with $|\lambda|\neq1$ and we denote by $\omega_0$ the $(1,1)$-form associated to the flat metric on $\mathds C^N$. Notice that  $N$ is allowed to be infinite and, when it is, we mean that the universal covering is $l^2(\mathds C)$.
We say that an lcK manifold $(M, \omega, \theta)$ admits {a (global)} {\em lcK immersion} into $(H^N_\lambda,\omega_H,\theta_0)$ if there exists a holomorphic map $f\colon M\to H^N_\lambda$ such that $f^*\omega_H=\omega$.
{In this case, we say that $\omega$ is \emph{Hopf-induced}.}
As a consequence of the Calabi local rigidity (see \cite{calabi-ann} or next section), when an lcK immersion exists it is unique up to unitary transformations of the ambient space, see Prop. \ref{prop:uniqueness}.

{
A straightforward obstruction to the existence of an lcK immersion is that Hopf-induced metrics are with proper potential and so they have the diffeomorphism type of Vaisman manifolds, (see Prop. \ref{prop:obstruction} and Rem. \ref{rem:obstruction}).
}
On the other hand, {an lcK manifold} $(M,\omega)$ is Hopf-induced if and only if a K\"ahler covering is induced by an immersion into $l^2(\mathds C)$ that preserves both the metric and the Lee form (see Prop. \ref{prop:equivalent-defi}). Applying Calabi's techniques, we are able to write such conditions in terms of the diastasis function (see \cite{calabi-ann} or next section for definitions and details). More precisely we have: 

\begin{theorem}[see Prop. \ref{prop:equivalent-defi-2}]\label{main}
Let $(M,\omega)$ be {an lcK manifold with proper potential}.
Then $(M,\omega)$ is Hopf-induced if the following conditions are fulfilled:
\begin{enumerate}
\item a K\"ahler covering $(\tilde M,\tilde\omega)$ of $(M,\omega)$ is induced by an immersion into $l^2(\mathds C)$;
\item the continuous extension of the diastasis function to the one--point completion of $(\tilde M,\tilde\omega)$ is an automorphic potential.
\end{enumerate}
\end{theorem}

Notice that in complex dimension two, one actually has to assume the Spherical Shell conjecture to hold true (see Sect. \ref{lcKnot} for notation and details).

We further observe that in the lcK context an analogue of Tian's Approximation Theorem recalled above does not hold true, as we prove in Prop. \ref{prop:non-tian}.

In the second part of this note, we focus on non-K\"ahler compact complex surfaces. The only ones admitting Vaisman metrics are either with non-negative Kodaira dimension or diagonal Hopf surfaces, as proven by F. Belgun \cite{belgun}. As for existence of lcK metrics, assuming the Spherical Shell conjecture to hold true, the situation is known thanks to \cite{belgun, brunella} and references therein. As observed by \cite{brunella, pontecorvo}, any compact complex surface admits a covering which is K\"ahler. In Sect. \ref{sec:surfaces}, we study lcK immersions for the Hermitian compact complex surfaces appearing in the literature \cite{gauduchon-ornea, parton-Ampa, belgun, tricerri, cordero-fernandez-deleon, vaisman-polito, tricerri}, we apply Theorem \ref{main} to the Vaisman ones and we also study the existence of a K\"ahler immersion of the covering when they are not Hopf-induced. In particular, linear Hopf surfaces with the Gauduchon-Ornea metric provides the only examples of lcK immersion we have, and only when the eigenvalues of the generator of the action group {coincide}.
Properly elliptic surfaces with the Vaisman metric by Belgun do not admit an lcK immersion, but they do at the level of the K\"ahler covering.
Similarly, also both the Vaisman metric on the Kodaira surfaces, and the (non-Vaisman) lcK metric by Tricerri on the Inoue-Bombieri surfaces are induced by K\"ahler metric immersed into $l^2(\mathds C)$, but they are not Hopf-induced.

The paper is organized as follows. In the next section we recall the definition of diastasis function and summarize what we need about Calabi's work on K\"ahler immersions of K\"ahler manifolds into the complex Euclidean space. In Sect. \ref{lcKiso} we state equivalent definitions of lcK immersions, and study their properties in relation with the diastasis function and the Lee form. We also investigate approximation of Vaisman manifolds by Hopf-induced metrics. Finally, Sect. \ref{sec:surfaces} is devoted to the case of lcK surfaces described above.

\bigskip

\noindent{\sl Acknowledgments.}
The authors are very grateful to Andrea Loi and Liviu Ornea for the interest in their work and for all the stimulating conversations.
{
We would also like to thank the anonymous referees for having pointed out some mistakes in a previous version and for their very useful suggestions.
}

\section{K\"ahler immersions of K\"ahler manifolds into the complex Euclidean space}
In the groundbreaking work \cite{calabi-ann}, E. Calabi states an algebraic criterion to test the existence of a K\"ahler immersion of a real analytic K\"ahler manifold into a finite or infinite dimensional complex space form, {\itshape i.e.} a simply connected K\"ahler manifold with constant holomorphic sectional curvature. We summarize in this section his criterion in the case when the ambient space is the complex Euclidean space endowed with the flat metric $g_0$, that will be the case we consider in the next sections.
Recall that, when the dimension is infinite, one takes $l^2(\mathds C)$ as ambient space. For details about K\"ahler immersions of K\"ahler manifolds into complex space forms and related issues, we refer the reader to \cite{loi-zedda} and references therein.

Consider a K\"ahler manifold $(M,g)$ and let $\Phi\colon U\to \mathds{R}$ be a K\"ahler potential for $g$, {\itshape i.e.} if we denote by $\omega$ the K\"ahler form associated to $g$, then $\omega|_U=\frac{\sqrt{-1}}2\partial\bar\partial \Phi$. Fix local coordinates $z$ on $U$. Since the pull-back of the flat metric $g_0$ through a holomorphic map must be a real analytic K\"ahler metric, we take $g$ to be real analytic. Then, we can duplicate the variables $z$ and $\bar z$ and consider the analytic continuation $\tilde \Phi(z, \bar z')$ of $\Phi(z,\bar z)=\Phi(z)$ in a neighbourhood $W$ of the diagonal $\Delta\subset U\times \bar U$. The {\em diastasis function} for $g$ on $W$ is given by:
$$
D(z,z')=\tilde \Phi(z, \bar z)+\tilde \Phi(z', \bar z')-\tilde \Phi(z, \bar z')-\tilde \Phi(z', \bar z).
$$
\begin{ex}\label{diastflat}\rm
The diastasis of the flat metric $g_0$ on $\mathds{C}^N$, $N\leq+\infty$, reads:
$$
D^0(z,z')=\|z-z'\|^2.
$$
\end{ex}
Once one of its two entries is fixed, the diastasis is a K\"ahler potential for $g$ with the following key property:
\begin{prop}[Hereditary property {\cite[Prop. 6]{calabi-ann}}]\label{induceddiast}
Let $(M,g)$ be a real analytic K\"ahler manifold and let $f\colon U\to \mathds{C}^N$, $N\leq+\infty$, be a K\"ahler map, {\itshape i.e.} $f^*g_0=g$. Then, if we denote by $D^0$ the diastasis function of $g_0$ and by $D$ that of $g$, we have:
$$
D(z,z')=D^0(f(z), f(z')).
$$
\end{prop} 
Let us denote by $D_0(z)=D(z,0)$ the diastasis for $g$ centered at the origin of the chosen coordinate system and let $(a_{jk})$ be the matrix of coefficients in the power series expansion:
\begin{equation}\label{diastexp}
D_0(z)=\sum_{j,k=1}^\infty a_{jk}z^{m_j}\bar z^{m_k}.
\end{equation}
where the $m_j$s are $n$-uples of integers, $n$ being the dimension of $M$, arranged in lexicographic order. 
\begin{rem}\label{charact}
In general two K\"ahler potentials differ by the addition of the real part of a holomorphic function, {\itshape i.e.} if $\varphi$ is a K\"ahler potential and $h$ is a holomorphic function, then $\varphi'=\varphi+h+\bar h$ is a K\"ahler potential as well. Among the other K\"ahler potentials, the diastasis function is characterized by presenting in each coordinates system a power expansion \eqref{diastexp} with $a_{0k}=a_{j0}=0$, or in other words, it does not have terms in $z$ or $\bar z$ alone.
\end{rem}
\begin{definition}\rm
A K\"ahler manifold $(M,g)$ is {\em resolvable of rank $N$} at $p\in M$ if, in a coordinates system $z$ centered at $p$, the matrix $(a_{jk})$ is semipositive definite of rank $N$. 
\end{definition}
Calabi proved that if a K\"ahler manifold is resolvable of rank $N$ at $p$ then it also is of the same rank at any other point {\cite[Th. 4]{calabi-ann}}. The resolvability of a real analytic K\"ahler manifold $(M,g)$ is actually equivalent to the existence of a local K\"ahler immersion. Calabi's criterion can be stated as follows:

\begin{theorem}[Calabi's local criterion {\cite[Th. 3]{calabi-ann}}]\label{localcriterion}
A K\"ahler manifold $(M,g)$ admits a local K\"ahler immersion of a neighborhood $U$ of a point $p$ into $\mathds{C}^N$, $N\leq +\infty$, if and only if $(M,g)$ is resolvable of rank at most $N$ at $p$. Further, if the rank is exactly $N$, then the immerson is {\em full}, {\itshape i.e.} the image of $U$ is not contained in any totally geodesic complex submanifold of $\mathds{C}^N$.
\end{theorem}

Further, when a K\"ahler immersion exists, it is unique up to unitary transformation of the ambient space:

\begin{theorem}[Calabi's rigidity {\cite[Th. 2]{calabi-ann}}]\label{Calabirig}
Let $f_1, f_2 :(M, g)\rightarrow \mathds C^N$, $N\leq+\infty$, be two full K\"ahler immersions.
Then there exists a unitary transformation $U$ of $\mathds C^N$ such that $f_2=U\circ f_1$.
\end{theorem}

When the manifold is simply connected, one can glue together all the local immersions to get a global one. When $M$ is not simply connected, the existence of a local K\"ahler immersion implies the existence of a global one $f\colon \tilde M\to \mathds{C}^N$ for its universal covering $\tilde M$. This map $f$ descends to $M$ when, for any $p\in M$, the analytic extension of the diastasis $D_p$ centered at $p$ is single valued. The global Calabi's criterion can be stated as follows:
\begin{theorem}[Calabi's global criterion {\cite[Th. 6]{calabi-ann}}]\label{thm:calabi-global}
A real analytic K\"ahler manifold $(M,g)$ admits a K\"ahler immersion into $\mathds{C}^N$, $N\leq+\infty$, if and only if the following conditions are fulfilled:
\begin{enumerate}
\item $(M,g)$ is resolvable of rank at most $N$;
\item for any $p\in M$, the analytic extension of the diastasis $D_p$ is single valued.
\end{enumerate}
\end{theorem}

\section{Isometric immersions of lcK manifolds in Hopf manifolds}\label{lcKiso}

\subsection{Notation on lcK structures}\label{lcKnot}
Let $M$ be a compact complex manifold, of complex dimension $n$. We recall that a {\em locally conformally K\"ahler} (shortly, {\em lcK}) structure on $M$ is given by a Hermitian metric $g$ with associated $(1,1)$-form $\omega$ such that $d\omega-\theta\wedge\omega=0$ with $d\theta=0$; here, the $1$-form $\theta$ is called the {\em Lee form}. Equivalently, an lcK structure corresponds to a covering $\tilde M \to M$ where $\tilde M$ is endowed with a K\"ahler metric $\tilde\omega$, globally conformal to the pull-back of $\omega$ itself, such that the deck transformation group $\Gamma$ acts by holomorphic homoteties, with factors given by the character $\chi \colon \Gamma \to \mathds R^{>0}$.

An lcK structure $(\omega,\theta)$ is said to be {\em with potential} if there exists a K\"ahler covering $(\tilde M, \tilde\omega)$ admitting a positive automorphic global potential, namely, $\tilde\omega=dd^c\psi$ where $\psi>0$ is a positive function such that $\gamma^*\psi=\chi(\gamma)\psi$ for any $\gamma\in\Gamma$.
Up to small deformations of $(\omega,\theta)$ in the $\mathcal{C}^\infty$-topology, we can assume that the potential $\psi$ is also {\em proper}, see \cite[page 93]{ornea-verbitsky-JGP}.
This last condition is the same as asking that the {\em lcK rank}, namely, the rank of $\mathrm{im}\,\chi$, is $1$, see \cite[Claim 2.8]{ornea-verbitsky-JGP}.

An lcK structure is called {\em Vaisman} if the associated Lee form is parallel with respect to the Levi Civita connection of $g$ itself. Equivalently, up to small deformations, they admit a K\"ahler covering being a Riemannian cone over a Sasaki manifold (see \cite{ornea-verbitsky-JGP} and references therein for dimension greater than two and \cite{belgun} for Vaisman compact complex surfaces). Observe that Vaisman manifolds are with potential, given by {$\|\vartheta^{\sharp_{\tilde g}}\|_{\tilde g}^2$, where $\tilde g$ is the metric associated to $\tilde \omega$}.

We recall that primary {\em Hopf manifolds} are quotient of $\mathds{C}^{n}\setminus \{0\}$ by a free action of the cyclic group $\Gamma = \langle \gamma \rangle \simeq \mathds{Z}$, where $\gamma$ acts by holomorphic contractions preserving $0$ (that is, its linearization in $0$ is invertible with eigenvalues of norm less than $1$). If $\gamma$ is linear, we say that the Hopf manifold is {\em linear}; the special case $\gamma=\lambda\cdot \mathrm{id}$ for $\lambda\in\mathds C$, $|\lambda|<1$ will be referred as {\em classical} (or {\em simple elliptic} following the notation for surfaces).
Linear Hopf manifolds are lcK with proper potential. {\em Diagonal Hopf manifolds}, namely, linear Hopf manifolds whose fundamental group is generated by a semisimple matrix, are in fact Vaisman.

In particular, primary {\em Hopf surfaces} \cite{kodaira-Surfaces} are obtained by $\gamma(z,w)=(\alpha z+\lambda w^m, \beta w)$ where $\alpha,\beta\in\mathds{C}\setminus\{0\}$, $\lambda\in\mathds{C}$, $m\in\mathds{N}\setminus\{0\}$, with $|\alpha|<1$, $|\beta|<1$, $\lambda (\alpha-\beta^m)=0$. They are Vaisman if and only if $\lambda=0$ (known as {\em class 1}) \cite[Theorem 1]{belgun}; also in case $\lambda\neq0$ (known as {\em class 0}), they admit a lcK metric \cite[Theorem 1]{gauduchon-ornea}.

In  \cite{ornea-verbitsky-MathAnn2, ornea-verbitsky-Cortona}, L. Ornea and M. Verbitsky prove the following embedding theorem:

\begin{theorem}[{\cite[Theorem 3.4]{ornea-verbitsky-MathAnn2}, \cite[Theorem 5.6]{ornea-verbitsky-Cortona}}]
Any compact lcK manifold with potential of complex dimension at least $3$ can be holomorphically embedded into a Hopf manifold. A compact Vaisman manifold of complex dimension at least $3$ can be holomorphically embedded into a diagonal Hopf manifold. If the Spherical Shell conjecture holds true, then the same conclusion holds for complex dimension $2$.
\end{theorem}

Here, the Spherical Shell conjecture refers to the statement that any class $\text{VII}_0$ surface with $b_2>0$ is a Kato surface.

\subsection{Infinite dimensional Hopf manifold}
In view of the Ornea and Verbitsky embedding theorem in lcK geometry, \cite[Theorem 3.4]{ornea-verbitsky-MathAnn2}, we consider as ambient space the {\em $N$-dimensional classical Hopf manifold} {of dimension $N\in\mathds N$} defined as
$$ H^N_\lambda := (\mathds C^N \setminus \{0\}) / \mathds Z , $$
where $\mathds Z$ is generated by $\lambda \cdot \mathrm{id}$ where $\lambda\in\mathds C\setminus\{0\}$ and $|\lambda|\neq1$. We endow it with the lcK metric
$$ g_H := \frac{1}{\|z\|^2}g_0 , $$
where $g_0$ is the flat metric on $\mathds C^N$; we denote by {$\omega_H=\|z\|^{-2}dd^c\|z\|^2$ and $\omega_0=dd^c\|z\|^2$} the associated $(1,1)$-forms. Its Lee form is
$$ \theta_0 \stackrel{\text{loc}}{=} d\log\|z\|^{-2} . $$
{
The lcK metric is in fact Vaisman with proper automorphic potential $\|z\|^2=D_0^{\mathds C^N}(z)$ equals to the diastasis centered in the origin.
}
Following Calabi's picture, we allow $N$ to be infinite: in this case,
$$ H^\infty_\lambda := (l^2(\mathds C)\setminus\{0\}) / \mathds Z , $$
{
where $l^2(\mathds C)$ is the space of sequences $z=(z_j)_j \subset \mathds C$ with finite $l^2$-norm, namely $\|z\|:=(\sum_{j} |z_j|^2)^{\frac{1}{2}}<+\infty$,
and $g_0$ is the flat metric on $l^2(\mathds C)$, which has associated $(1,1)$-form $\omega_0=dd^c\|z\|^2$.
Moreover, the homothety character $\chi_0 \colon \mathds Z \to \mathds R^{>0}$ has image $\langle \lambda \rangle$ of rank $1$.
In this sense, we consider $H^\infty_\lambda$ as an lcK manifold with "proper" potential.
}

\subsection{Equivalent definitions of lcK immersion}
Let $(M,\omega)$ be an lcK manifold, of complex dimension $\dim_{\mathds C} M \geq2$, with Lee form $\theta$, and denote by $(\tilde M, \tilde \omega)$ a K\"ahler covering, {\itshape i.e.} $\tilde \omega=\exp(-h)\omega$ is a K\"ahler metric, with the pull-back $\tilde\theta=d h$ on $\tilde M$. 

\begin{definition}
An lcK manifold $(M,\omega)$ admits a {\em local lcK immersion} into the classical Hopf manifold $(H^N_\lambda,\omega_H)$, where $N\leq+\infty$, if, for any point $x\in M$, there exist a neighbourhood $U\ni x$ and a holomorphic map $F \colon U \to H^N_\lambda$ such that $F^*\omega_H = \omega$.
We call an lcK immersion {\em full} if the dimension of the ambient space can not be lowered.
\end{definition}

Note that the lcK metric determines the Lee form, and that the following objects are essentially equivalent: the Lee form, its local potential, the monodromy group of the covering.
Then the existence of lcK immersions can be stated in equivalent ways:
\begin{prop}
Let $(M,\omega)$ be an lcK manifold, of complex dimension $\dim M \geq2$, with Lee form $\theta$. If $(M,\omega)$ admits a local lcK immersion into $(H^N_\lambda,\omega_H)$, where $N\leq+\infty$, then $F$ is a holomorphic immersion such that $F^*\theta_0=\theta$.
\end{prop}
\begin{proof}
We notice that, $F$ being an isometry, it is also an immersion.
Moreover, the condition $F^*\omega_H=\omega$ assures also $F^*\theta_0=\theta$. Indeed, we have:
$$ d \omega = d F^* \omega_H = F^* d \omega_H  = F^*(\theta_0 \wedge \omega_H) = F^*\theta_0 \wedge \omega, $$
and on the other hand:
$$ d \omega = \theta \wedge \omega .$$
By comparing the two expressions, we get:
$$ (F^*\theta_0 - \theta ) \wedge \omega=0, $$
and the claim follows since $\omega\wedge\_\colon \wedge^1M \to \wedge^3M$ is injective when the complex dimension of $M$ is at least $2$.
\end{proof}

\begin{prop}\label{prop:equivalent-defi}
An lcK manifold $(M,\omega)$ {with proper potential}, of complex dimension $\dim M \geq2$ and with Lee form $\theta$, admits a local lcK immersion into $(H^N_\lambda,\omega_H)$, where $N\leq+\infty$, if and only if a minimal K\"ahler covering $(\tilde M, \tilde \omega)$ of $(M, \omega)$ admits a local K\"ahler immersion $\tilde F$ into $(\mathds C^N, \omega_0)$ and $\tilde F^* \tilde \theta_0 = \tilde \theta$, where $\tilde\theta_0$ and $\tilde\theta$ are the pull-back of the Lee forms on the coverings.
{Moreover, if such an $\tilde F \colon \tilde M \to \mathds C^N$ is a global K\"ahler immersion, then $F \colon M \to H^N_\lambda$ is a global lcK immersion, too.}
\end{prop}
\begin{proof}
We prove that a local lcK immersion of $(M,\omega)$ into $(H^N_\lambda,\omega_H)$ induces a local K\"ahler immersion of a minimal K\"ahler covering $(\tilde M, \tilde \omega)$ into $(\mathds C^N, \omega_0)$, which in fact is equivariant with respect to the actions of the deck transformation group on $\tilde M$ and of $\mathds Z=\langle \lambda \cdot \mathrm{id} \rangle$ on $\mathds C^N$.
Indeed, take $F_1\colon M \supseteq U \to \mathds C^N$ any lift of the map $F\colon M \supseteq U \to H^N_\lambda = \mathds C^N \slash \mathds Z$; observe that two lifts differ by the action of $\mathds Z=\langle \lambda \cdot \mathrm{id} \rangle$.
We notice that, on $\tilde M$, one has $\tilde\theta=d h$. Then the condition $F^*\theta_0=\theta$ can be rewritten as $d(\tilde F^*\log\|z\|^{-2}-h)=0$, or equivalently, $\tilde F^*\|z\|^{-2}=\exp h$ up to multiplicative constants. Then:
$$ \tilde F^*\omega_0 = \tilde F^*(\|z\|^2 \omega_H) = \tilde F^*\|z\|^2 \cdot F^*\omega_H \\
= \exp(-h)\omega = \tilde \omega, $$
up to multiplicative constants.
We notice that, if $F$ is a local K\"ahler immersion of $(\tilde M,\tilde \omega)$ into $(\mathds C^N,\omega_0)$, then $\sqrt{c}F$, for $c\in\mathds R^{>0}$ constant, is a local K\"ahler immersion of $(\tilde M, c^{-1}\tilde\omega)$ into $(\mathds C^N,\omega_0)$.
Then $F$ induces a local K\"ahler immersion $\tilde F$ of the K\"ahler covering into $\mathds C^N$ with the flat metric $\omega_0$.

Conversely, we prove that, if a minimal K\"ahler covering $(\tilde M, \tilde \omega)$ of $(M, \omega)$ admits a local {(respectively, global)} K\"ahler immersion $\tilde F$ into $(\mathds C^N, \omega_0)$ such that $\tilde F^*\tilde\theta_0=\tilde\theta$, then
{
we can make $\tilde F$ equivariant with respect to a suitable action on $\mathds C^N$, which is in fact a multiple of the identity, and then it induces a local (respectively, global) lcK immersion of $(M, \omega)$ into $(H^N_\lambda,\omega_H)$, for some $\lambda$.
Indeed, consider the homothety characters $\chi_M \colon \pi_1(M) \supseteq \mathrm{Aut}_M\tilde M \to \mathds{R}^{>0}$, where $\mathrm{Aut}_M\tilde M$ is the deck transformations group of the covering $\tilde M \to M$.
The homothety character is determined by the Lee form, see {\itshape e.g.} \cite[Theorem 3.1]{parton-vuletescu}, or by its lift: indeed, $\chi_M(\gamma) = \exp \int_\gamma \theta = \exp \int_{\tilde\gamma} \tilde \theta$ where $\tilde\gamma$ is a lift of $\gamma \in \mathrm{Aut}_M\tilde M \subseteq \pi_1(M)$ to $\tilde M$.
Take $\gamma \in \pi_1(M)$ a generator of $\mathrm{Aut}_M\tilde M$. We have
$\lambda^{-2} := \chi_M(\gamma) = \exp \int_{\tilde \gamma} \tilde F^* \tilde\theta_0=\exp \int_{\tilde F\circ\tilde\gamma}\tilde\theta_0$ and also $\lambda^{-2} = \exp \int_{z}^{\lambda\cdot z} d\log \|z\|^{-2} = \exp \int_{z}^{\lambda\cdot z} \tilde\theta_0$ that we expect to be a generator for $\chi_0(H^N_\lambda)$ the image of the homothety character of $H^N_\lambda$.
Both $\tilde F$ and $\lambda \tilde F \circ \tilde\gamma$ are isometric immersions into $(\mathbb C^N,g_0)$, then they differ by a unitary transformation of $\mathbb C^N$ thanks to Calabi's rigidity Th. \ref{Calabirig}. Then, up to unitary transformations, we have that $(\lambda\tilde F) \circ \tilde \gamma = \tilde F$. Then we get that $\tilde F\circ \tilde \gamma=\zeta\circ \tilde F$ for the trasformation $\zeta=\lambda^{-1} \cdot {\rm id}$ of the whole $\mathds C^N$.

We conclude by showing that $F^*\omega_H=\omega$. Indeed, we have $\tilde \theta_0 = d \log \|z\|^{-2}$ and, say, $\tilde \theta = d h$; then the condition $\tilde F \tilde \theta_0 = \tilde \theta$ rewrites as $F^*\|z\|^{-2}=\exp h$ up to multiplicative constants.
}
Finally, we compute
$$ F^*\omega_H= F^*(\|z\|^{-2}\omega_0) = F^*\|z\|^{-2} \cdot F^*\omega_0 = F^*\|z\|^{-2} \cdot \tilde\omega = \exp(\log F^*\|z\|^{-2}-h) \omega =\omega , $$
concluding the proof.
\end{proof}

In the sight of Prop. \ref{prop:equivalent-defi}, an lcK immersion is full in $H^N_\lambda$ if and only if the K\"ahler immersion of the K\"ahler covering is full in $\mathds C^N$.

\subsection{Properties of lcK immersions}

If an lcK immersion exists, then it is essentially unique. More precisely we have:
\begin{prop}\label{prop:uniqueness}
Local lcK immersions are determined up to unitary transformations of $\mathds C^N$, $N\leq+\infty$.
\end{prop}

\begin{proof}
Let $F_1$ and $F_2$ be two local lcK immersions of $M$ into $H^N_\lambda$. We can lift and interpret them as local K\"ahler immersions $\tilde F_1$ and $\tilde F_2$ for a minimal K\"ahler covering $\tilde M$ into $\mathds C^N$. Thanks to the Calabi's rigidity Th. \ref{Calabirig}, K\"ahler immersions are defined up to unitary transformations and translations. In this case, $\tilde F_1$ and $\tilde F_2$ being equivariant, we just have a unitary transformation $\tilde U$ of $\mathds C^N$ such that $\tilde F_2 = \tilde U \circ \tilde F_1$. If $\pi_M \colon \tilde M \to M$ and $\pi_0 \colon \mathds C^N \to H^N_\lambda$ denote the coverings, then $F_2 \circ \pi_M = \pi_0 \circ \tilde F_2 = \pi_0 \circ \tilde U \circ \tilde F_1 = U \circ \pi_0 \circ \tilde F_1 = U \circ F_1 \circ \pi_M$, whence $F_2 = U \circ F_1$ where $U$ is induced by $\tilde U$ on the quotient.
\end{proof}

A first straightforward obstruction for the existence of lcK immersions is the following:

{
\begin{prop}\label{prop:obstruction}
An lcK manifold $(M,\omega)$ admitting {a global lcK immersion} into an infinite dimensional classical Hopf manifold admits a proper potential.
\end{prop}

\begin{proof}
We have $\omega = F^*\omega_H = F^* ( \|z\|^{-2} dd^c \|z\|^2 ) = F^* \|z\|^{-2} dd^c F^*\|z\|^2 = \psi^{-1}dd^c\psi$ where $\psi:=F^*\|z\|^2=\|F\|^2$ is a positive automorphic global potential.

Moreover, the potential is proper. More precisely, we prove that $(M,\omega)$ has lcK rank $1$, that is, the image of its homothety character has rank $1$. This is equivalent to the potential being proper thanks to \cite[Claim 2.4]{ornea-verbitsky-JGP}.
By Prop. \ref{prop:equivalent-defi}, $F$ induces a local K\"ahler immersion $\tilde F$ of the K\"ahler covering $\tilde M$ into $\mathds C^N$ which is equivariant. 
Thus, for any $\gamma\in \mathrm{Aut}_M\tilde M$, we can choose $\zeta \in \mathrm{Aut}_H {\mathds C^N}$ such that $\zeta\circ \tilde F=\tilde F\circ \gamma$. Since $\zeta^*\omega_0=\lambda^k \cdot \omega_0$ for some $k\in \mathds N$ where $\lambda \in \mathrm{im}\, \chi_0$ is a generator of the image of the character of $H$, we have:
$$
\tilde \omega=\tilde F^*\omega_0= \lambda^{-k} \tilde F^*(\zeta^*\omega_0)=\lambda^{-k} (\zeta\circ \tilde F)^*\omega_0=\lambda^{-k} \gamma^*\tilde F^*(\omega_0)=\lambda^{-k}\gamma^*\tilde \omega,
$$
whence
$$
\chi_M(\gamma)=\frac{\gamma^*\tilde\omega}{\tilde\omega}=\lambda^k.
$$
This proves that the image of $\chi_M$ is a subgroup of the image of $\chi_0$, whence the lcK rank of $M$ is $1$.
\end{proof}

\begin{rem}\label{rem:obstruction}
In general, we cannot use the immersion $F$ to conclude that a Hopf-induced lcK metric is Vaisman too, unless the immersion is totally-geodesic. On the other hand, by \cite[Theorem 2.1]{ornea-verbitsky-IMRN} (and \cite[Proposition 2.13]{ornea-verbitsky-JGP}), compact complex manifolds endowed with lcK metrics with (proper) potential can be smally deformed to Vaisman manifolds. In particular, {\itshape compact manifolds admitting Hopf-induced metrics are diffeomorphic to Vaisman manifolds}.
\end{rem}

\begin{rem}
We thank the anonymous referee for having drawn our attention to the result in \cite[Proposition 6.5]{verbitsky}. It states that {\itshape a compact complex submanifold of a Vaisman manifold is itself Vaisman}. The argument in \cite{verbitsky} shows that the restriction of the Lee vector field $\theta^\sharp$ (namely, the metric dual of the Lee form) to the submanifold is actually tangent to the submanifold, as a consequence of the formula $d^c\theta=-\omega+\theta\wedge J\theta$, which holds more in general for lcK structures with potential, \cite[Lemma 5.1]{ornea-verbitsky-JGP09}.
\end{rem}
}

For an lcK manifold $(M,\omega)$ with proper potential and of dimension $n\geq 3$, L. Ornea and M. Verbitsky \cite{ornea-verbitsky-MathAnn} show that the metric completion $\tilde M_c$ of a minimal K\"ahler covering $\tilde M$ is Stein, and $\tilde M_c \setminus \tilde M$ is just one point, that in the sequel we will denote by $p$. For compact complex surfaces, the same result holds under the assumption of the Spherical Shell conjecture (see \cite[Th. 5.6]{ornea-verbitsky-Cortona}).
The following result provides necessary and sufficient conditions to the existence of an lcK immersion in terms of the function $D^{\tilde M_c}_p$ defined as the continuous extension of the diastasis function $D^{\tilde M}_q$ to the one point completion $\tilde M_c$ of $\tilde M$. More precisely if $\tilde F\!:\tilde M\rightarrow \mathds C^N\setminus \{0\}$ is a K\"ahler immersion (and we denote in the same way also its continuous extension to the metric completions) then we define:
\begin{equation}\label{dpc}
\lim_{q\rightarrow p}D^{\tilde M}_q=\lim_{q\rightarrow p}F^*||z-F(q)||^2=F^*\lim_{q\rightarrow p}||z-F(q)||^2=F^*||z||^2=:D^{\tilde M_c}_p,
\end{equation}
where we used the definition of the diastasis of $\mathds C^N$ Ex. \ref{diastflat} and the Hereditary property Prop. \ref{induceddiast}.
Note that if the metric completion is smooth, $D^{\tilde M_c}_p$ is exactly the diastasis function centered at $p$; but in general $\tilde M_c$ is not smooth, unless $M$ is a Hopf manifold.

\begin{prop}\label{prop:equivalent-defi-2}
Let $(M,\omega)$ be {an lcK} manifold with proper potential, of complex dimension $\dim_{\mathds C} M \geq2$, and let $\theta$ be the Lee form.
In case $\dim_{\mathds C}M=2$, assume also the Spherical Shell conjecture to hold true.
Let $(\tilde M, \tilde\omega)$ be a minimal K\"ahler covering and $(\tilde M_c, \tilde\omega_c)$ be its metric completion, obtained by adding the point $p$.

The following are equivalent:
\begin{enumerate}
\item $(M,\omega)$ admits a local lcK immersion into an $N$-dimensional classical Hopf manifold;
\item $(\tilde M, \tilde\omega)$ is resolvable of rank at most $N$ and the function $D_p^{\tilde M_c}$ given by \eqref{dpc}, satisfies $\theta\stackrel{\text{loc}}{=}-d\log D_p^{\tilde M_c}$;
\item $(\tilde M, \tilde\omega)$ is resolvable of rank at most $N$ and the function $D_p^{\tilde M_c}$, given by \eqref{dpc}, yields a local automorphic potential for $\tilde\omega$.
\end{enumerate}

{
Moreover, if for any $q\in \tilde M$, the analytic extension of the diastasis $D^{\tilde M}_q$ is single valued, (for example, if $\tilde M$ is simply-connected,) then $(M,\omega)$ is Hopf-induced, that is, it admits a global lcK immersion into an $N$-dimensional classical Hopf manifold.	
}
\end{prop}

\begin{proof}
In Prop. \ref{prop:equivalent-defi}, we have seen that the existence of a local lcK immersion of $M$ into $H^N_\lambda$, $N\leq+\infty$, is equivalent to the existence of a local K\"ahler immersion of $\tilde M$ into $\mathds C^N$ that preserves the pull-backs of the Lee forms, that is, $F^*\tilde\theta_0=\tilde\theta$.

The condition on the resolvability is equivalent to the existence of a local K\"ahler immersion $\tilde F$ of $\tilde M$ into $\mathds C^N\setminus\{0\}$ thanks to Calabi's criterion Th. \ref{localcriterion}.
Since $\tilde\theta_0=d\log\|z\|^{-2}
$, by \eqref{dpc} the condition $\tilde F^*\tilde\theta_0=\tilde\theta$ is then equivalent to $\tilde\theta=-d\log D^{\tilde M_c}_p$.

Finally, we prove the equivalence of the third statement.
Clearly, if $\tilde\omega=dd^c\Phi$ has automorphic potential $\Phi=D_p^{\tilde M_c}$, then $\tilde\theta=-d\log D_p^{\tilde M_c}$, and we are back to the second statement. Conversely, since $\omega_H = \|z\|^{-2} dd^c \|z\|^2 $ and again by \eqref{dpc}, the condition $F^*\omega_H=\omega$ is equivalent to $\omega=(D_p^{\tilde M_c})^{-1}dd^cD_p^{\tilde M_c}$.

{
Moreover, if for any $q\in \tilde M$, the analytic extension of the diastasis $D^{\tilde M}_q$ is single valued, then Calabi's global criterion in Th. \ref{thm:calabi-global} applies for $\tilde M$, that is, $\tilde F \colon \tilde M \to \mathds C^N$ is a global K\"ahler immersion. Then it induces a global lcK immersion $F \colon M \to H^N_\lambda$ thanks to the observation in Prop. \ref{prop:equivalent-defi}.
}
\end{proof}

Finally, we show that an analogue of the Tian approximation Theorem \cite{tian-jdg} does not hold in the lcK context.

\begin{prop}\label{prop:non-tian}
In general, it is not true that Vaisman metrics can be approximated by Hopf-induced metrics.
\end{prop}

\begin{proof}
We consider the following counterexample.
Consider the Vaisman metrics on diagonal Hopf surfaces with eigenvalues $\alpha\neq \beta$ such that $|\alpha|=|\beta|$ described in Sect. \ref{sec:surfaces}. We will prove in Prop. \ref{prop:hopf-surface-ricoprimento} and Cor. \ref{cor:hopf-surface-quoziente} that a K\"ahler covering $(\tilde M, \tilde \omega)$ admits a K\"ahler immersion into $l^2(\mathds C)\setminus\{0\}$, but such immersion does not descend to the quotient.
Observe that in this case the metric completion is actually smooth.
We show now that such a metric can not be approximated by Hopf-induced metrics.

Due to Prop. \ref{prop:equivalent-defi-2}, the fact that the immersion does not descend to the quotient means that the automorphic potential of $\tilde \omega$ differs from the diastasis function $D^{\tilde \omega_c}_p$ of the metric completion $(\tilde M^c, \tilde \omega_c)$ by the (non-vanishing) real part of a holomorphic function, $p$ being the completing-one-point.

Assume by contradiction that $\omega_m \to \omega$ where $\omega_m$ are Hopf-induced metrics, {\itshape i.e.} $\omega_m = F_m^* \omega_H$ for some $F_m \colon M \to H^\infty_\lambda$ holomorphic immersions. Then we have $\tilde F_m \colon \tilde M \to l^2(\mathds C) \setminus\{0\}$ holomorphic immersions such that $\tilde F^*\omega_0 = \tilde \omega_m = dd^c \tilde F_m^*\|z\|^2$. On one hand, we have $F_m^* \theta_0 \to \theta$, where $\theta$ is the Lee form of $\omega$; on the other hand, on the K\"ahler coverings, the pull-backs of the Lee forms satisfy $\tilde\theta_m = - d \log \tilde F_m^* \|z\|^2$ and $\tilde\theta = - d \log \Phi$, where $\Phi$ is the automorphic potential of $\tilde \omega$. So we have $F_m^*\|z\|^2 \to \Phi$, up to additive constants. Now we notice that, since $F_m$ are holomorphic maps, the projections of $F_m^*\|z\|^2$ onto the space of real parts of holomorphic function is zero. But the same does not hold true for $\Phi$, because, by assumption, the automorphic potential is not the diastasis of the one-point metric completion of $(\tilde M, \tilde \omega)$ (see Rem. \ref{charact}).
\end{proof}

\section{Hopf-induced Vaisman surfaces}\label{sec:surfaces}

\subsection{The Vaisman question}
In \cite[Th. 1]{belgun}, extending \cite{vaisman-polito, tricerri, gauduchon-ornea}, F. Belgun proves that a compact complex non-K\"ahler surface $X$ is Vaisman if and only if it is an elliptic surface ({\itshape i.e.} a properly elliptic surface, a primary or secondary Kodaira surface, or an elliptic Hopf surface), or it is a Hopf surface of class $1$. Except for some Inoue surfaces (those of type $S^+_{n;p,q,r;u}$ corresponding to a parameter $u\not\in\mathds R$), he constructs lcK metrics on any other compact complex surfaces with first Betti number $b_1$ odd (namely, non-K\"ahler) and Euler characteristic $\chi=0$ (which is the obstruction to be Vaisman).
The missing case of class $\text{VII}^+_0$ surfaces reduces to Kato surfaces under the Spherical Shell conjecture. In \cite[Th. 1]{brunella}, extending previous works by C. LeBrun, and A. Fujiki and M. Pontecorvo, M. Brunella proves that any Kato surface admits an lcK metric. So, under the Spherical Shell conjecture, this provides a complete answer to the Vaisman question \cite[p. 122]{vaisman-polito} on the existence of lcK metrics on compact complex surfaces. In fact, notice that even the non-lcK Inoue surfaces do admit a covering being K\"ahler, see \cite[p. 77]{brunella} and \cite[Rem. 4.5]{pontecorvo}. We refer to \cite{pontecorvo} for a review on the Vaisman question.

\subsection{Hopf-induced lcK metrics}
Of course, any diagonal Hopf surface and any compact complex surface with non-negative Kodaira dimension admits an lcK metric and an lcK immersion into a classical Hopf manifold: indeed, they are Vaisman by \cite{belgun}, and \cite{ornea-verbitsky-MathAnn2, ornea-verbitsky-Cortona} provides such a metric as pull-back of a holomorphic embedding.

In this section, we explicitly study lcK immersions for the known Vaisman metrics in literature, paying also attention to whether a K\"ahler immersion of the K\"ahler covering exists or not. What we prove is the following:
\begin{description}
 \item[Diagonal Hopf surfaces] diagonal Hopf surfaces are Vaisman \cite{belgun}, and Vaisman structures are provided by P. Gauduchon and L. Ornea in \cite{gauduchon-ornea}; for classical Hopf surfaces, in particular, more metrics are provided by M. Parton in \cite{parton-Ampa}. The Gauduchon-Ornea metric turns out to be induced by an lcK immersion into a classical Hopf manifold if and only if the eigenvalues of the generator of the fundamental group {are equal, $\alpha=\beta$; the condition that their norms are equal $|\alpha|=|\beta|$} is necessary for the existence of a K\"ahler immersion of the K\"ahler covering. On classical Hopf manifolds, also the Parton metrics are induced by an immersion, at the level both of the lcK manifold and of the K\"ahler covering.
 \item[Properly elliptic surfaces] properly elliptic surfaces are Vaisman \cite{belgun}. We consider the Vaisman structure constructed in \cite[Sect. 3]{belgun}. We show that, the K\"ahler immersion into $l^2(\mathds{C})$ exists at the level of the K\"ahler covering, but it is not equivariant. Thus the Belgun metric on properly elliptic surfaces is not Hopf-induced.
 \item[Kodaira surfaces] Kodaira surfaces are Vaisman \cite{belgun}. We consider the Vaisman structures constructed in \cite{cordero-fernandez-deleon, vaisman-polito}. At the level of the K\"ahler covering, they are induced by the flat metric on $l^2(\mathds C)$, through a full K\"ahler immersion. But the explicit immersion fixes the origin, up to translation, so it does not descend to the quotient.
 \item[Inoue surfaces] Inoue surfaces are not Vaisman \cite{belgun}. We consider Inoue-Bombieri surfaces of type $S_M$ endowed with the lcK metric by F. Tricerri \cite{tricerri}. Its K\"ahler covering has a local full immersion into $l^2(\mathds C)\setminus\{0\}$ with the flat metric, but the manifold itself does not admit an lcK immersion into the Hopf manifold, {since they are not diffeomorphic to Vaisman manifolds by \cite{belgun}}.
\end{description}
Summarizing, among these examples, Hopf surfaces are the only ones with Hopf-induced lcK metrics, and we notice that the dimension of the ambient space is finite, but not {\itshape a priori} bounded.

Observe that the existence of a K\"ahler immersion of the K\"ahler covering into $l^2(\mathds C)$ is  intersting in its own sake. For example, due to the maximum principle, a complete K\"ahler submanifold of a complex Eucidean space must be noncompact and the only known homogeneous K\"ahler manifolds admitting such an immersion are the complex flat space, the complex hyperbolic space, and  products of them (see \cite{discala-ishi-loi}).

\subsection{Diagonal Hopf surfaces}\label{sec:hopf}
We study when linear Hopf surfaces endowed with the Gauduchon-Ornea metric admit an isometric immersion into an infinite dimensional classical Hopf manifold, and also when their K\"ahler coverings do into $l^2(\mathds C)$.

More precisely, we consider the compact complex surfaces $H_{\alpha,\beta}=\mathds C^2\setminus\{0\} / \mathds Z$ where $\mathds Z$ is generated by $\gamma(z,w)=(\alpha z, \beta w)$ where $\alpha,\beta\in\mathds{C}\setminus\{0\}$ with $|\alpha|\geq|\beta| > 1$, and endowed  with the lcK metric constructed in \cite{gauduchon-ornea}. The K\"ahler covering $\mathds C^2\setminus\{0\}$ has (automorphic) potential $\Phi_{\alpha,\beta}$ that satisfies the equation \cite[Prop. 1, Eq. (10)]{gauduchon-ornea}:
\begin{equation}\label{eqphi}
 |z_1|^2 \Phi_{\alpha,\beta}^{-\frac{2\log|\alpha|}{\log|\alpha|+\log|\beta|}}(z_1,z_2) + |z_2|^2 \Phi_{\alpha,\beta}^{-\frac{2\log|\beta|}{\log|\alpha|+\log|\beta|}}(z_1,z_2) = 1 ,
 \end{equation}
that is, see {\itshape e.g.} \cite[page 126]{ornea-survey},
$$ \Phi_{\alpha,\beta}(z_1,z_2) = \exp\left( \frac{(\log|\alpha|+\log|\beta|)\tau}{2\pi} \right) $$
where $\tau$ is the unique solution of the equation:
$$ \frac{|z_1|^2}{\exp(\tau\log|\alpha|\pi)} + \frac{|z_2|^2}{\exp(\tau\log|\beta|\pi)} = 1 . $$
Note that the above equation is rotation invariant and thus its solution $\tau$ also is. It follows that  $\Phi$ is rotation invariant as well, namely $\Phi(z_1,z_2)=\Phi(|z_1|^2, |z_2|^2)$.

We prove that $H_{\alpha,\beta}$ is Hopf-induced if and only if $|\alpha|=|\beta|$, which in particular includes the classical case $\alpha=\beta$.

We begin with the following lemma.
\begin{prop}\label{prop:hopf-surface-ricoprimento}
Consider on $\mathds{C}^2\setminus\{0\}$ the K\"ahler metric $\omega_{\alpha,\beta}$ induced by the rotation invariant potential $\Phi$ which satisfies \eqref{eqphi}.
Then $(\mathds{C}^2\setminus\{0\},\omega_{\alpha,\beta})$ admits a local K\"ahler immersion into $l^2(\mathds{C})$ if and only if $|\alpha|=|\beta|$.
Moreover, the immersion is full in $\mathds C^2$.
\end{prop}

\begin{proof}
If $|\alpha|=|\beta|$, then \eqref{eqphi} gives:
$$
\Phi = |z_1|^2+|z_2|^2,
$$
and $\omega_{\alpha,\beta}$ is the flat metric on $\mathds{C}^2\setminus \{0\}$, which trivially admits a K\"ahler immersion into $\mathds{C}^2$. 

In order to deal with the "only if" part, let us write:
$$
a=\frac{2\log|\alpha|}{\log|\alpha|+\log|\beta|},\qquad b=\frac{2\log|\beta|}{\log|\alpha|+\log|\beta|}.
$$
By computing:
$$
 \left. \frac{\partial^{2j}}{\partial  z_1^j\partial\bar z_1^j}\left(|z_1|^2\Phi (|z_1|^2,|z_2|^2)^{-a}+|z_2|^2\Phi(|z_1|^2,|z_2|^2)^{-b}\right)\right|_{z_1=0,z_2=s} = 0,
$$
for $s\in \mathds{C}\setminus \{0\}$, one gets:
$$
\left. \frac{\partial^2}{\partial z_1\partial \bar z_1}\Phi(|z_1|^2,|z_2|^2)\right|_{z_1=0,z_2=s}=\frac{\Phi(0,|s|^2)^{b-a+1}}{|s|^{2}b},
$$
 and:
\begin{equation}\label{derivate}
\left.\frac{\partial^{2j}}{\partial z_1^j\partial \bar z_1^j}\Phi(|z_1|^2,|z_2|^2)\right|_{z_1=0,z_2=s}=\frac{\Phi(0,|s|^2)^{j(b-a)+1}}{|s|^{2j}b^{j}}\prod_{k=1}^{j-1}(kb+1-ja),\nonumber
\end{equation}
for $j\geq 2$.
We claim that for any value of $a$ and $b$ such that $a\neq b$, there exists $j$ big enough such that $\left.\frac{\partial^{2j}}{\partial z_1^j\partial \bar z_1^j}\Phi(|z_1|^2,|z_2|^2)\right|_{z_1=0,z_2=s}<0$.
If the claim holds true, then for any $|\alpha|\neq |\beta|$,  by Calabi's criterion, $(\mathds C^2\setminus\{0\},\omega_{\alpha,\beta})$, does not admit a local K\"ahler immersion into $l^2(\mathds{C})$. In fact, since $\Phi$ depends only on the modules of the variables, the matrix $(a_{jk})$ given by \eqref{diastexp} is diagonal and $\left.\frac{\partial^{2j}}{\partial z_1^j\partial \bar z_1^j}\Phi(|z_1|^2,|z_2|^2)\right|_{z_1=0,z_2=s}$ are eigenvalues.
In order to prove the claim, observe that $\left.\frac{\partial^{2j}}{\partial z_1^j\partial \bar z_1^j}\Phi(|z_1|^2,|z_2|^2)\right|_{z_1=0,z_2=s}<0$ if and only if $\prod_{k=1}^{j-1}(kb+1-ja)<0$. Without loss of generality, we can assume, by symmetry, that $a>1>b$. The claim holds for any positive even integer $j$, by noting that the sequence of factors is increasing in $k$ and that the last term is $(j-1)b+1-ja=-j(a-b)+(1-b)$. 
\end{proof}

{When the K\"ahler immersion exists, then it is the identity on the image up to rotations. Then it is equivariant with respect to $(z,w) \stackrel{\gamma}{\mapsto} (\alpha z, \beta w)$ and $(z,w)\mapsto (\lambda z,\lambda w)$ if and only if $\alpha=\beta$ and $|\alpha|=|\beta|=\lambda$:}
\begin{cor}\label{cor:hopf-surface-quoziente}
Let $H_{\alpha,\beta}=\mathds C\setminus\{0\} / \mathds Z$ be a linear Hopf manifold, where $\mathds Z$ is generated by $\gamma(z,w) = (\alpha z, \beta w)$.
Then $H_{\alpha,\beta}$ admits a local lcK immersion into $H^\infty_\lambda$ if and only if {$\alpha=\beta$, and in this case $\lambda=|\alpha|=|\beta|$.}
Moreover, the immersion is full in $H^2_{\lambda}$.
\end{cor}

As second example, we consider classical Hopf surfaces endowed with the following metrics constructed by M. Parton in \cite{parton-Ampa}.

For $k\in \mathds{R}^+$, consider the family of Vaisman metrics on $\mathds C^2\setminus\{0\}$ given by:
\begin{equation}
\begin{split}
\omega_k=\frac{1}{(|z_1|^2+|z_2|^2)^2}\left[(\right.&k|z_1|^2+|z_2|^2)dz_1\wedge d\bar z_1+(k-1)z_2\bar z_1 dz_1\wedge d\bar z_2\\
&\left.+(k-1)z_1\bar z_2 dz_2\wedge d\bar z_1+(|z_1|^2+k|z_2|^2)dz_2\wedge d\bar z_2\right].
\end{split}\nonumber 
\end{equation}
They descend to the quotient $H_{\alpha,\beta}$ as soon as $\alpha=\beta$, generalizing the metrics $\omega_{\alpha,\alpha}$.
The metric:
$$
\tilde\omega_k=(|z_1|^2+|z_2|^2)^k\omega_k,
$$
is K\"ahler on $\mathds{C}^2\setminus\{0\}$ with K\"ahler potential $\Phi_k=\frac1k(|z_1|^2+|z_2|^2)^k$.
When $k$ is a positive integer, $\tilde \omega_k$ is defined on the whole $\mathds{C}^2$ and the potential $\Phi_k$ is actually the diastasis function centered at the origin for $\tilde\omega_k$.
It follows that the diastasis is an automorphic potential for $\omega_k$. Further, it is not hard to see that, for any positive integer value of $k$, then $(\mathds{C}^2,\tilde\omega_k)$ is resolvable of rank $k+1$ and thus it admits a full K\"ahler immersion into $(\mathds{C}^{k+1},\omega_0)$. By Prop. \ref{prop:equivalent-defi}, $(\mathds{C}^2\setminus\{0\}/\mathds{Z},\omega_k)$  admits a full lcK immersion into $H^{k+1}_\lambda$ with $\lambda=k^k$.

Further, the K\"ahler map can be written explicitely by:
$$
\tilde F\colon \mathds{C}^2\to \mathds{C}^{k+1},\quad \tilde F=(\tilde F_0,\dots, \tilde F_k), \ \ \tilde F_j=\sqrt{\frac{1}{k} \binom{k}{j}}z_1^{k-j}z_2^{j}.
$$ 

\begin{rem}
We notice that the dimension of the ambient space for lcK immersions does not depend just on the holomorphic structure of the lcK manifold. In fact, $(H_{\alpha,\alpha},\omega_k)$ admits an lcK immersion into a Hopf manifold of dimension $k+1$.
\end{rem}

\subsection{Properly elliptic surfaces}
Let $D:=\{(x, y)\in \mathds C^2 \;:\; \ \Im(x\bar y)>0\}$. By \cite[Prop. 2]{belgun} any minimal properly elliptic surface $X$ is diffeomorphic as a complex manifold to the quotient of $D$ by the free and proper action of a discrete subgroup $\Gamma$ of $G:=\mathds C^*\cdot \mathrm{SL}(2,\mathds R)\subset GL(2,\mathds C)$. We consider the lcK form $\omega$ on $X$ given by $\omega=\frac12{\Im}(x\bar y)dd^c({\Im}(x\bar y)^{-1})$ and its K\"ahler covering is $(D,\tilde \omega)$, where $\tilde \omega$ is described by the (globally) defined K\"ahler potential $\Phi:={\Im}(x\bar y)^{-1}$.
We have the following:
\begin{prop}\label{pes}
The K\"ahler manifold $(D,\tilde \omega)$ admits a K\"ahler immersion into $l^2(\mathds C)$.
\end{prop}
\begin{proof}
A local K\"ahler immersion $\varphi \!:D\setminus\{z_1=\sqrt{-1}z_2\}\rightarrow l^2(\mathds C)$ of $(D,\tilde \omega)$ into $l^2(\mathds C)$ is given by:
$$
\varphi(z_1, z_2)= \frac{1}{z_1-\sqrt{-1}z_2}\left(\dots , \left(\frac{z_1+\sqrt{-1}z_2}{z_1-\sqrt{-1}z_2}\right)^j,\dots \right),
$$
for $j\in \mathds N$.
Indeed:
\begin{equation}
\begin{split}
\varphi^*(\omega_0)=&\frac{\sqrt{-1}}{2}\partial\bar\partial \left[ \frac{1}{|z_1-\sqrt{-1}z_2|^2}\sum_{j=0}^{\infty}|\frac{z_1+\sqrt{-1}z_2}{z_1-\sqrt{-1}z_2}|^{2j}\right]\\
=&
\frac{\sqrt{-1}}{2}\partial\bar\partial \left[ \frac{1}{\left|z_1-\sqrt{-1}z_2\right|^2}\cdot\frac{1}{1-\left|\frac{z_1+\sqrt{-1}z_2}{z_1-\sqrt{-1}z_2}\right|^{2}}\right]\\
=&\frac{\sqrt{-1}}{2}\partial\bar\partial\left[ \frac{1}{\left|z_1-\sqrt{-1}z_2\right|^2-\left|z_1+\sqrt{-1}z_2\right|^2}\right]\\
=&\frac{\sqrt{-1}}{2}\partial\bar\partial \left[\Im (z_1\bar z_2)^{-1}\right].
\end{split}\nonumber
\end{equation}
Observe that, since $D$ is simply connected, the local K\"ahler immersion can be extended to a global one.
\end{proof}

Then we have:
\begin{cor}
$X$ does not admit an lcK embedding into a classical Hopf manifold.
\end{cor}
\begin{proof}
By Calabi's rigidity Theorem (Th. \ref{Calabirig} above) a global K\"ahler immersion $\psi \!:D\rightarrow l^2(\mathds C)$ is given in $D\setminus\{z_1=\sqrt{-1}z_2\}$ by $U\circ \varphi$ where $\varphi$ is the map given in Prop. \ref{pes} and $U$ is a unitary transformation of $l^2(\mathds C)$. It follows that each entry $\psi_j$ of $\psi$ is a linear combination of $\left(\frac{z_1+\sqrt{-1}z_2}{z_1-\sqrt{-1}z_2}\right)^j$, $j\in \mathds N$, multiplied by $\frac{1}{z_1-\sqrt{-1}z_2}$. Conclusion follows by noticing that  $\left(\frac{z_1+\sqrt{-1}z_2}{z_1-\sqrt{-1}z_2}\right)^j$ (and thus any linear combination of it) is invariant by the action of $\mathds Z^*\cdot {\rm Id}\subset \Gamma$ while $\frac{1}{z_1-\sqrt{-1}z_2}$ is not.
\end{proof}

\subsection{Kodaira surfaces}
It is an open problem to determine which {\em nilmanifolds} (namely, compact quotients of connected simply connected nilpotent Lie group by cocompact discrete subgrous) admit lcK metrics. L. Ugarte \cite{ugarte} conjectured that necessarily the group is isomorphic to the product of the Heisenberg group times $\mathds R$; G. Bazzoni \cite{bazzoni} proved that the conjecture is true under the further assumption of Vaisman. The lowest dimensional case, that is, when the group is $\mathrm{Heis}(3;\mathds R) \times \mathds R$, the corresponding manifolds are the (primary) {\em Kodaira surfaces} \cite{kodaira-Surfaces}. They are characterized by Kodaira dimension $0$, odd first Betti number, trivial canonical bundle.

Vaisman metrics on (primary) Kodaira surfaces are explicitly studied in \cite{cordero-fernandez-deleon, vaisman-polito} and are defined as follows.
With respect to coordinates:
$$ (x,y,t,u)=\left( \left(\begin{matrix}1 & x & t \\ & 1 & y \\ & & 1 \end{matrix}\right) , u \right) \in \mathrm{Heis}(3;\mathds R) \times \mathds R, $$
consider the holomorphic coordinates on the universal cover:
$$ z=\frac{1}{\sqrt{2}}(x+\sqrt{-1}y), \qquad  w = t - \frac{1}{2}xy -\frac{\sqrt{-1}}{4} (x^2+y^2)+\sqrt{-1}u . $$
The Vaisman lcK metric reads:
$$ \omega = 2 dz \wedge d\bar z + \frac{1}{2}(dw+\sqrt{-1}\bar z dz)\wedge (d\bar w-\sqrt{-1}z d\bar z) , $$
while on the K\"ahler covering one has the K\"ahler metric:
$$ \tilde \omega =  e^{u/2} \left( 2 dz \wedge d\bar z + \frac{1}{2}(dw+\sqrt{-1}\bar z dz)\wedge (d\bar w-\sqrt{-1}z d\bar z) \right) . $$

A K\"ahler potential for $\tilde \omega$ is given by $\Phi(z,w)=8 e^{u/2}$, as can be easily seen by noticing that:
$$
u=\frac12|z|^2-\frac{\sqrt{-1}}2(w-\bar w),
$$
and thus:
$$
\frac{\partial^2}{\partial z\partial\bar z}e^{u/2}=\frac1{16}e^{u/2}(|z|^2+4),\qquad \frac{\partial^2}{\partial w\partial\bar w}e^{u/2}=\frac1{16}e^{u/2},
$$
$$
\frac{\partial^2}{\partial z\partial\bar w}e^{u/2}=\frac{\sqrt{-1}}{16}e^{u/2}\bar z,\qquad \frac{\partial^2}{\partial w\partial\bar z}e^{u/2}=-\frac{\sqrt{-1}}{16}e^{u/2} z.
$$
Note that the lcK metric has automorphic potential: $\omega \stackrel{\text{loc}}{=} 4e^{-u/2} dd^c(e^{u/2})$.

\begin{prop}\label{lemmma:kodaira-covering}
Let $\tilde \omega$ be the K\"ahler metric on $\mathds{C}^2$ defined by the K\"ahler potential $\Phi(z,w)=\exp(\frac14|z|^2-\frac{\sqrt{-1}}4(w-\bar w))$. Then $(\mathds{C}^2,\omega)$ admits a global full K\"ahler immersion into $l^2(\mathds{C})$.
\end{prop}
\begin{proof}
It is easy to see that the map $F\colon \mathds{C}^2\to l^2(\mathds{C})$ given by:
\begin{equation}\label{kodmap}
F(z,w)=\left(\dots, \frac{z^j}{2^j\sqrt{j!}}\exp\left(-\frac{\sqrt{-1}}{4}w\right),\dots\right),
\end{equation}
satisfies $F^*\omega_0=\tilde \omega$. In fact, one has:
\begin{equation}
\begin{split}
||F_j||^2=&\exp\left(-\frac{\sqrt{-1}}{4}w\right)\exp\left(\frac{\sqrt{-1}}{4}\bar w\right)\frac14\sum_{j=0}^\infty\frac{|z|^{2j}}{{4^jj!}}\\
=&\exp\left(-\frac{\sqrt{-1}}{4}(w-\bar w)\right)\exp\left(\frac{|z|^{2}}{{4}}\right)\\
=& \Phi(z,w),
\end{split}\nonumber
\end{equation}
as wished.
\end{proof}

\begin{cor}\label{kodno}
A K\"ahler map $F\colon \mathds{C}^2\to l^2(\mathds{C})$, $F^*\omega_0=\tilde \omega$, does not induce any lcK map $f\colon\mathrm{Heis}(3;\mathds R) \times \mathds R\to l^2(\mathds C)$.
\end{cor}

\begin{proof}
By Calabi's rigidity Theorem (see Th. \ref{Calabirig} above), the K\"ahler map $F$ given by \eqref{kodmap}, is unique up to unitary transformations and translations of $l^2(\mathds C)$.
It follows that a translation $T$ such that $T\circ F$ descends to the quotient does not exist.
\end{proof}

\subsection{The Inoue-Bombieri surface}\label{sec:iowa}
{
We show an example of an lcK manifold whose K\"ahler covering has a local full immersion into $l^2(\mathds C)\setminus\{0\}$, but not admitting an lcK immersion into the Hopf manifold, since the diffeomorphism type is different from Vaisman manifolds.
}

It is given by the {\em Inoue-Bombieri surface} \cite{inoue}. Inoue surfaces are surfaces of class VII ({\itshape i.e.} Kodaira dimension $-\infty$ and first Betti number $1$), with second Betti number equal to zero and with no holomorphic curves. Their universal covering is $\mathds C\times\mathds{H}$, where $\mathds{H}$ denotes the hyperbolic plane.
They are divided into three families: $S_M$, $S^+_{N,p,q,r;\mathbf{t}}$, and $S^-_{N,p,q,r}$.

We consider here the case $S_M$.
More precisely, consider coordinates $(z=x_1+\sqrt{-1}y_1, w=x_2+\sqrt{-1}y_2)\in\mathds C\times\mathds{H}$ with $y_2>0$.
Fix a matrix $M=(M_{jk})_{j,k}\in \mathrm{SL}(3;\mathds{Z})$ with a real eigenvalue $\lambda>1$ and complex non-real eigenvalues $\mu$ and $\bar\mu$; denote by $(\ell_1,\ell_2,\ell_3)$ an eigenvector for $\lambda$, and by $(m_1,m_2,m_3)$ an eigenvector for $\mu$.
Define
$$ S_M := \left. (\mathds C\times\mathds{H}) \middle\slash \Gamma \right. , $$
where $\Gamma=\left\langle f_0,f_1,f_2,f_3 \right\rangle$ is the group generated by
\begin{eqnarray*}
f_0(z,w) &:=& \left( \mu z, \lambda w \right) , \\
f_j(z,w) &:=& \left( z+m_j, w+\ell_j \right) ,
\end{eqnarray*}
varying $j\in\{1,2,3\}$.
The action of $\Gamma$ on $\mathds C\times\mathds{H}$ is fixed-point free and properly discontinuos with compact quotient, so $S_M$ is a compact complex manifold, which has a structure of torus-bundle over $\mathds S^1$.

We endow $S_M$ with the Tricerri metric \cite{tricerri}:
$$ \omega_T := \sqrt{-1} y_2 dz \wedge d \bar z+\frac{\sqrt{-1}}{4 y_2^2} dw\wedge d\bar w, $$
which is an lcK metric on $S_M$ with Lee form $\theta=\frac{d y_2}{y_2}$:
$$ d\omega_T = \frac{d y_2}{y_2} \wedge \omega_T . $$

We lift to the covering $\tilde S_M$ of $S_M$ where $\log y_2$ is defined, so the Lee form $\theta$ becomes globally exact. That is,  we consider
$$ \tilde S_M := \left. ( \mathds C \times \mathds{H} ) \middle\slash \langle f_1, f_2, f_3 \rangle \right. . $$
The deck transformation group of the covering $\tilde S_M \to S_M$ is then identified with the cyclic group $\langle f_0 \rangle$.

Then, on $\tilde S_M$, we have the K\"ahler metric
\begin{eqnarray*}
\tilde \omega_T &:=& \exp(-\log y_2) \omega_T \\
&=& \sqrt{-1} dz \wedge d \bar z+\frac{\sqrt{-1}}{4 y_2^3} dw\wedge d\bar w .
\end{eqnarray*}
The deck transformation group $\langle f_0 \rangle$ acts on $\tilde \omega_T$ by homotheties:
\begin{eqnarray*}
f_0^* \tilde\omega_T &=&
\sqrt{-1} |\mu|^2 dz \wedge d \bar z+\frac{\sqrt{-1}}{4 y_2^3} \frac{1}{\lambda} dw\wedge d\bar w \\
&=& |\mu|^2 \tilde\omega_T ;
\end{eqnarray*}
more precisely, the group of homothety factors is the cyclic group $\mathds{Z} \simeq \langle |\mu|^{2} \rangle$.

For simplicity, to write down the diastasis function for $\tilde\omega_T$, we move to the disc model $\mathds{B}$ instead of half-plane model $\mathds{H}$ for the hyperbolic plane, by the transformation $\mathds{B}\ni w \mapsto \hat w = \frac{w+\sqrt{-1}}{\sqrt{-1}w+1} \in \mathds{H}$ and its inverse $\mathds{H}\ni \hat w \mapsto w = \frac{\hat{w}-\sqrt{-1}}{-\sqrt{-1}\hat{w}+1} \in \mathds{B}$. In these coordinates $\hat w \in \mathds{H}$, we have:
$$ \tilde \omega_T = \sqrt{-1} dz \wedge d \bar z + 4\sqrt{-1}\frac{|1+\sqrt{-1}\hat w|^2}{(1-|\hat w|^2)^3}d\hat w\wedge d\bar{\hat w} . $$ 
A K\"ahler potential for $\tilde \omega_T$, which is also the diastasis function centered at the origin, is given by:
$$
D_0(z,\hat w)=|z|^2+\frac{2(2+\sqrt{-1}(\hat w-\bar{\hat w}))|\hat w|^2}{1-|\hat w|^2} .
$$
The map $F\!:\mathds C\times\mathds B\rightarrow l^2(\mathds C)$ defined by:
\begin{equation}\label{eq:F-Inoue}
F(z,\hat w)=\left(z, \sqrt{2}\hat w, \sqrt{2}(\hat w^2-\sqrt{-1}\hat w),\dots,\sqrt{2}(\hat w^{j+1}-\sqrt{-1}\hat w^j),\dots, \right),
\end{equation}
is a full K\"ahler immersion, as it follows by:
\begin{equation}
\begin{split}
||F_j||^2=&|z|^2+2|\hat w|^2+2\sum_{j=1}^\infty|\hat w^{j+1}-\sqrt{-1}\hat w^j|^2\\
=&|z|^2+2|\hat w|^2+\frac{2|\hat w|^2|\sqrt{-1}-\hat w|^2}{1-|\hat w|^2}\\
=&|z|^2+\frac{2(2+\sqrt{-1}(\hat w-\bar{\hat w}))|\hat w|^2}{1-|\hat w|^2}.
\end{split}\nonumber
\end{equation}
Up to translate the image, we get a local K\"ahler immersion into $l^2(\mathds{C})\setminus\{0\}$.

We turn now at the problem of the existence of an lcK immersion of $S_M$ itself into some Hopf manifold.
{
We notice that this is impossible by Prop. \ref{prop:obstruction}, since the Inoue-Bombieri surface do not admit a proper potential since they are not diffeomorphic to Vaisman manifolds \cite{inoue, belgun}.
}
It follows also by looking at the explicit form of the immersion: the condition $F^*\theta_0=\theta$ (where $\theta_0$ is the Lee form of the Hopf manifold) yields $F^*\|z\|^{-2}=y_2$ up to additive constants; but $F$ is the identity on the first factor $\mathds C \subset \mathds C \times \mathds H$.

\begin{rem}
We ask whether the Inoue-Bombieri surfaces admit an lcK immersion into some other (possibly infinite dimensional) lcK non-Vaisman Hopf manifold.
For example, we ask whether there is any contraction $\varphi$ on $l^2(\mathds C)$ that makes $F$ in \eqref{eq:F-Inoue} equivariant and for which $l^2(\mathds C) / \mathds \langle \varphi \rangle$ is an lcK {manifold} with flat K\"ahler covering.
\end{rem}

\end{document}